\documentclass[11pt]{article}

\usepackage[T1]{fontenc}
\usepackage{lmodern}
\usepackage{amsmath,amssymb,amsthm,mathtools}
\usepackage{microtype}
\usepackage{tikz}
\usepackage[margin=1in]{geometry}
\usepackage{enumitem}
\usepackage[numbers,sort&compress]{natbib}
\usepackage[hidelinks]{hyperref}
\hypersetup{pdftitle={Graphs with Long Pseudosimilarity Chains under Consecutive Vertex Deletions},pdfauthor={Sergey Ivanov}}

\newtheorem{theorem}{Theorem}[section]
\newtheorem{lemma}[theorem]{Lemma}

\newtheorem{corollary}[theorem]{Corollary}
\theoremstyle{definition}
\newtheorem{definition}[theorem]{Definition}
\newtheorem{example}[theorem]{Example}
\newtheorem{problem}[theorem]{Problem}
\theoremstyle{remark}

\newcommand{\Aut}{\operatorname{Aut}}
\newcommand{\psd}{\operatorname{psd}}
\newcommand{\PS}{\operatorname{PS}}
\newcommand{\Z}{\mathbb{Z}}

\title{Graphs with Long Pseudosimilarity Chains under Consecutive Vertex Deletions}
\author{Sergey Ivanov}
\date{}

\begin{document}
\maketitle

\begin{abstract}
Pseudosimilar vertices are vertices in distinct automorphism orbits whose deletions produce isomorphic graphs. Classical work has studied the existence, group-theoretic origin, and construction of large sets of such vertices. We ask a different recursive question: how long can one repeatedly delete a vertex that is pseudosimilar at the moment of deletion? We define the pseudosimilarity depth of a graph and construct connected graphs in which this process continues through all but a sublinear number of vertices. A two-clock construction gives a square-root deficit uniformly in the order, while a Chinese-remainder construction with many cyclic clocks yields an infinite family of asymmetric graphs with only a polylogarithmic number of vertices left outside the active chain. The mechanism realizes pseudosimilarity by breaking a long hidden automorphism orbit and enlarging the break one vertex at a time. Thus pseudosimilarity can persist through an asymptotically full sequence of vertex deletions, even though every graph encountered in the main construction is asymmetric.
\end{abstract}

\section{Introduction}

For a finite simple graph $G$ and vertices $u,v\in V(G)$, call $u$ and $v$ \emph{similar} if some automorphism of $G$ maps $u$ to $v$. They are \emph{removal-similar} if
\[
    G-u \cong G-v,
\]
and \emph{pseudosimilar} if they are removal-similar but not similar. The distinction is elementary but important: two vertices can occupy genuinely different roles in the intact graph while becoming indistinguishable after either one is deleted.

Pseudosimilarity arose in the early reconstruction literature by Harary and Palmer who exhibited the phenomenon \cite{HararyPalmer1966}. Kimble, Schwenk, and Stockmeyer developed the subject systematically, constructing asymmetric graphs in which every vertex has a pseudosimilar mate and graphs with arbitrarily large mutually pseudosimilar sets \cite{KimbleSchwenkStockmeyer1981}. Godsil and Kocay showed that every finite pair of pseudosimilar vertices arises from a general ambient-automorphism construction \cite{GodsilKocay1982}; Kocay subsequently studied attachment operations and gave a group-theoretic treatment of mutually pseudosimilar sets \cite{KocayAttaching1984,Kocay1984}. Pseudosimilarity in trees was characterized by Kirkpatrick, Klawe, and Corneil \cite{KirkpatrickKlaweCorneil1983}. Large sets and related extremal constructions were studied further by Lauri \cite{Lauri1996,Lauri1997,Lauri2003}; a modern account appears in the monograph of Lauri and Scapellato \cite{LauriScapellato2016}. More broadly, quantitative reconstruction parameters ask how many vertex- or edge-cards suffice to determine a graph \cite{Lauri1993}; degree-associated versions additionally record the degree of the deleted vertex or edge. Deck-overlap questions instead ask how many cards two nonisomorphic graphs may share, and Ivanov recently showed that this overlap can be an arbitrarily large fraction of the full vertex deck \cite{Ivanov2026}.

The existing literature is primarily horizontal: how many vertices of one graph can be pseudosimilar, or how can a prescribed set of pseudosimilar vertices be constructed? Here we ask a vertical question. Suppose that $v$ is pseudosimilar to another vertex of $G$ and we delete $v$. Can the resulting card again contain a pseudosimilar vertex that can be deleted, and can this continue for many successive levels? A one-step precursor is already present in Lauri's reconstruction survey: for a pseudosimilar pair $u,v$, either a specific interchange-symmetry obstruction appears in $G-u-v$, or one of the cards contains a new pseudosimilar pair \cite{Lauri1993}. We are not aware of a parameter in the literature that measures how long this phenomenon can be iterated.

We introduce such a parameter. A deletion is called \emph{active} if the vertex being deleted has a pseudosimilar mate at that moment. The pseudosimilarity depth $\psd(G)$ is the maximum number of consecutive active deletions starting from $G$, and
\[
    \PS(n)=\max\{\psd(G): |V(G)|=n\}.
\]
The main result is that the depth can be asymptotically as large as the graph itself. More precisely, we prove that for all sufficiently large $n$,
\[
    \PS(n) \ge n-O(\sqrt n),
\]
so $\PS(n)/n\to1$. The graphs giving the basic bound are connected. We then refine the construction on an infinite sequence of orders: there are connected asymmetric graphs $G$ for which
\[
    \psd(G) \ge |V(G)|-
    O\!\left(\frac{(\log |V(G)|)^2}{\log\log |V(G)|}\right),
\]
and every intermediate graph in the active chain is asymmetric as well.

The mechanism is simple. We first build a small graph with a cyclic automorphism, which we call a \emph{clock}. A large family of additional vertices is indexed by positions on a hidden cyclic orbit. Deleting a consecutive interval of these vertices destroys every nontrivial automorphism, but the two vertices immediately adjacent to the gap remain removal-similar because deleting either one merely shifts the gap by one position. One clock gives a linear but small fraction of active deletions. Two relatively prime clocks encode a product number of positions with only a sum number of auxiliary vertices, giving the square-root deficit. Many clocks and the Chinese remainder theorem reduce the auxiliary part to polylogarithmic size.

\section{Active deletion chains}

We begin by fixing the recursive notion used throughout.

\begin{definition}
Let $G_0=G$. An \emph{active pseudosimilarity chain of length $k$} is a sequence of distinct vertices $v_1,\ldots,v_k$ such that, for each $j=0,\ldots,k-1$, if
\[
    G_j=G-\{v_1,\ldots,v_j\},
\]
then $v_{j+1}$ has a pseudosimilar mate in $G_j$. We set
\[
    \psd(G)=\max\{k: G\text{ has an active pseudosimilarity chain of length }k\}.
\]
For $n\ge1$, let
\[
    \PS(n)=\max_{|V(G)|=n}\psd(G).
\]
\end{definition}

The word active is important. If one only asks that a pseudosimilar pair survive somewhere while unrelated vertices are deleted, then arbitrary padding makes long chains immediate. In an active chain, the deleted vertex itself participates in the deletion ambiguity at every step.

The classical smallest graph with a pseudosimilar pair has eight vertices; the example goes back to Harary and Palmer and is recorded by Kimble, Schwenk, and Stockmeyer \cite{KimbleSchwenkStockmeyer1981}. Consequently,
\begin{equation}\label{eq:trivial-upper}
    \PS(n)\le n-7.
\end{equation}
A proof of this elementary bound is included in Appendix~\ref{app:upper-bound}. Our aim is to show that, asymptotically, it is nearly attainable.

The construction is based on a broken cyclic orbit. The following observation isolates the removal-similarity part of the argument.

\begin{lemma}[Sliding-gap lemma]\label{lem:sliding-gap}
Let $H$ have an automorphism $\rho$ and let $x_0,\ldots,x_{L-1}$ be vertices satisfying $\rho(x_j)=x_{j+1}$, with indices modulo $L$. For
\[
    G_i=H-\{x_0,\ldots,x_{i-1}\},\qquad 1\le i\le L-2,
\]
the vertices $x_i$ and $x_{L-1}$ are removal-similar in $G_i$.
\end{lemma}

\par\noindent The proof is given in Appendix~\ref{app:sliding-gap}.

The remaining task is to destroy all automorphisms that could exchange the two endpoints of the gap. Before presenting the cyclic gadget in general, we begin with its smallest instance.

\section{A 17-vertex example}

We first display the smallest instance of the construction.  Let the vertices $a_0,\ldots,a_5$ form a $6$-cycle, let $b_0,\ldots,b_5$ be independent with
\[
    N(b_i)=\{a_i,a_{i+1},a_{i+3}\},
\]
where subscripts are read modulo $6$, and add pendant vertices $x_0,\ldots,x_5$ with $x_i$ adjacent only to $a_i$.  Delete $x_0$ and call the resulting graph $G_1$.

\begin{example}[A 17-vertex chain]\label{ex:seventeen}
The graph $G_1$ has $17$ vertices and $29$ edges.  For $1\le i\le5$, set
\[
    G_i=G_1-\{x_1,\ldots,x_{i-1}\}.
\]
Then $G_1,G_2,G_3,G_4$ are connected and asymmetric, and in $G_i$ the vertices $x_i$ and $x_5$ are pseudosimilar for $i=1,2,3,4$.  Hence
\[
    G_1 \xrightarrow{-x_1} G_2
        \xrightarrow{-x_2} G_3
        \xrightarrow{-x_3} G_4
        \xrightarrow{-x_4} G_5
\]
is an active pseudosimilarity chain of length four, so $\psd(G_1)\ge4$.  This is the smallest instance of the present clock construction, not a claim of global minimality for depth four; the bound~\eqref{eq:trivial-upper} only forces order at least $11$.
\end{example}

\par\noindent The verification is given in Appendix~\ref{app:example-seventeen}.

Figure~\ref{fig:G17} shows $G_1$.  The six vertices $a_i$ form the outer cycle, the vertices $b_i$ form the inner layer, and the surviving $x$-vertices are the five pendant leaves.  The missing leaf at $a_0$ marks a break in a hidden cyclic orbit.  Each active deletion moves one endpoint of this break by one position; the resulting card is isomorphic to the card obtained by deleting the fixed vertex $x_5$, although no nontrivial automorphism survives in the current graph.

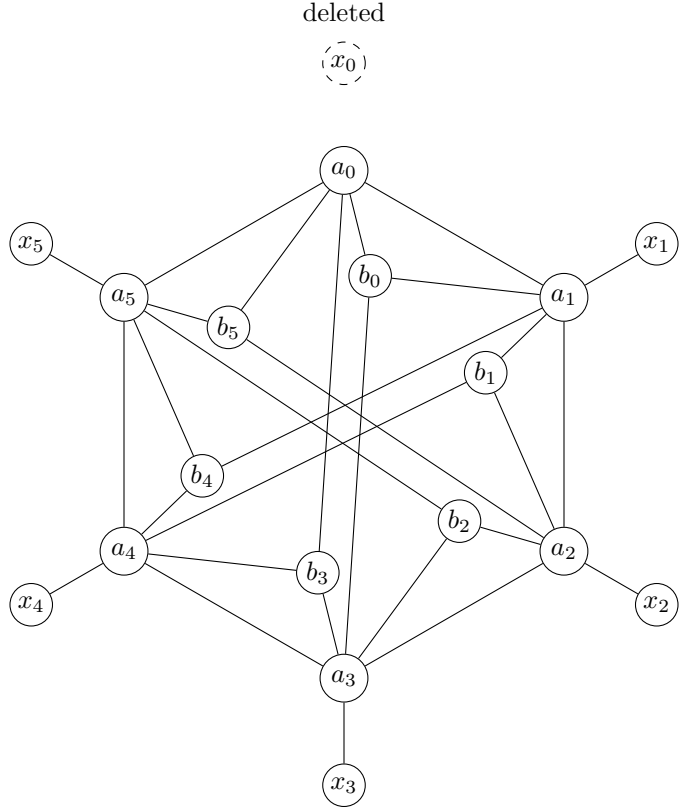
\begin{figure}[t]
\centering
\begin{tikzpicture}[scale=1.05, every node/.style={font=\small}]
  \tikzset{
    av/.style={circle, draw, inner sep=1.5pt, minimum size=18pt},
    bv/.style={circle, draw, inner sep=1.2pt, minimum size=16pt},
    xv/.style={circle, draw, inner sep=1.2pt, minimum size=16pt},
    del/.style={circle, draw, dashed, inner sep=1.2pt, minimum size=16pt}
  }
  \foreach \i/\ang in {0/90,1/30,2/-30,3/-90,4/-150,5/150}{
    \node[av] (a\i) at (\ang:3.2) {$a_{\i}$};
    \node[bv] (b\i) at ({\ang-10}:1.90) {$b_{\i}$};
  }
  \foreach \i/\ang in {1/30,2/-30,3/-90,4/-150,5/150}{
    \node[xv] (x\i) at (\ang:4.55) {$x_{\i}$};
  }
  \node[del] (x0) at (90:4.55) {$x_{0}$};
  \node at (90:5.2) {deleted};
  \foreach \u/\v in {0/1,1/2,2/3,3/4,4/5,5/0}{
    \draw (a\u) -- (a\v);
  }
  \foreach \i in {1,2,3,4,5}{
    \draw (a\i) -- (x\i);
  }
  \foreach \i/\j/\k in {0/1/3,1/2/4,2/3/5,3/4/0,4/5/1,5/0/2}{
    \draw (b\i) -- (a\i);
    \draw (b\i) -- (a\j);
    \draw (b\i) -- (a\k);
  }
\end{tikzpicture}
\caption{The $17$-vertex starting graph $G_1$.  The dashed vertex $x_0$ has already been deleted and is shown only to mark the initial gap.  Starting from $G_1$, delete $x_1,x_2,x_3,x_4$ in that order.  Immediately before deleting $x_i$ ($i=1,2,3,4$), the active pseudosimilar pair is $(x_i,x_5)$: deleting $x_i$ or $x_5$ gives isomorphic cards, while the current graph $G_i$ is asymmetric.  Thus the four successive pairs are $(x_1,x_5)$, $(x_2,x_5)$, $(x_3,x_5)$, and $(x_4,x_5)$.}
\label{fig:G17}
\end{figure}

The example contains the full mechanism used below.  A cyclic automorphism exists before the first leaf is removed; deleting a consecutive interval destroys that automorphism while retaining card isomorphisms at the two endpoints of the gap.

\section{One cyclic clock}

The preceding example is the case $m=6$ of the following general construction. Fix an integer $m\ge6$. Define a graph $Q_m$ with vertices
\[
    A=\{a_0,\ldots,a_{m-1}\},\qquad
    B=\{b_0,\ldots,b_{m-1}\},
\]
with indices modulo $m$. The vertices in $A$ induce the cycle
\[
    a_0a_1\cdots a_{m-1}a_0,
\]
the set $B$ is independent, and
\[
    N(b_i)=\{a_i,a_{i+1},a_{i+3}\}.
\]
We call $Q_m$ the $m$-clock.

\begin{lemma}\label{lem:clock-aut}
For every $m\ge6$,
\[
    \Aut(Q_m)\cong C_m,
\]
generated by the rotation $a_i\mapsto a_{i+1}$, $b_i\mapsto b_{i+1}$.
\end{lemma}

\par\noindent The proof is given in Appendix~\ref{app:clock-aut}.

We now attach one vertex to each clock position. Let $H_m$ be obtained from $Q_m$ by adding an independent set
\[
    X=\{x_0,\ldots,x_{m-1}\}
\]
and the edges $x_i a_i$. The clock rotation extends to $H_m$ by $x_i\mapsto x_{i+1}$.

For $1\le i\le m-1$, set
\[
    G_i=H_m-\{x_0,\ldots,x_{i-1}\}.
\]

\begin{theorem}[One-clock chain]\label{thm:one-clock}
For every $m\ge6$ and every $1\le i\le m-2$, the vertices $x_i$ and $x_{m-1}$ are pseudosimilar in $G_i$. Moreover, each $G_i$ is asymmetric. Consequently the connected graph $G_1$ has
\[
    |V(G_1)|=3m-1,
    \qquad
    \psd(G_1)\ge m-2.
\]
\end{theorem}

\par\noindent The proof is given in Appendix~\ref{app:one-clock}.

The one-clock construction gives linear depth, but only about one third of its vertices are active. To reduce the overhead, we encode each active vertex simultaneously on two clocks.

\section{Two clocks and a square-root deficit}

Let $p,q\ge6$ be distinct coprime integers and put
\[
    L=pq.
\]
Take disjoint copies $Q_p$ and $Q_q$ of the clocks above. Their $A$-vertices will be denoted by $a^{(p)}_j$ and $a^{(q)}_j$. Add an independent set
\[
    X=\{x_0,\ldots,x_{L-1}\},
\]
and join $x_k$ to
\[
    a^{(p)}_{k\bmod p}
    \quad\text{and}\quad
    a^{(q)}_{k\bmod q}.
\]
Call the resulting graph $H_{p,q}$. Simultaneously rotating both clocks by one step and sending $x_k$ to $x_{k+1}$ is an automorphism of order $L$.

For $1\le i\le L-1$, define
\[
    G_i=H_{p,q}-\{x_0,\ldots,x_{i-1}\}.
\]

\begin{theorem}[Two-clock chain]\label{thm:two-clock}
For $p,q\ge6$ distinct and coprime, every graph $G_i$, $1\le i\le L-1$, is connected and asymmetric. For $1\le i\le L-2$, the vertices $x_i$ and $x_{L-1}$ are pseudosimilar in $G_i$. In particular,
\[
    |V(G_1)|=pq+2p+2q-1
\]
and
\[
    \psd(G_1)\ge pq-2.
\]
\end{theorem}

\par\noindent The proof is given in Appendix~\ref{app:two-clock}.

Taking consecutive integers already gives the advertised square-root deficit.

\begin{corollary}\label{cor:sqrt-sequence}
For every $m\ge6$ there is a connected asymmetric graph $G$ on
\[
    n_m=m^2+5m+1
\]
vertices with
\[
    \psd(G)\ge m^2+m-2=n_m-4m-3.
\]
Consequently, along this sequence of orders,
\[
    \psd(G)=n_m-O(\sqrt{n_m}).
\]
\end{corollary}

\par\noindent The proof is given in Appendix~\ref{app:sqrt-sequence}.

The same estimate can be made uniform in the order without sacrificing connectedness. For a graph $F$ and $t\ge0$, let $K_t\vee F$ denote the join obtained by adding $t$ new vertices forming a clique and adjacent to every vertex of $F$.

\begin{lemma}[Join padding]\label{lem:join-padding}
If $u,v$ are pseudosimilar in $F$, then they are pseudosimilar in $K_t\vee F$ for every $t\ge0$, provided $F$ has no universal vertex. More generally, if no graph appearing along an active pseudosimilarity chain in $F$ has a universal vertex, then the chain remains active after joining the same $K_t$ at every stage.
\end{lemma}

\par\noindent The proof is given in Appendix~\ref{app:join-padding}.

Our clock graphs have no universal vertices, and the same remains true throughout the active chain. We obtain the following global estimate.

\begin{theorem}[Uniform asymptotic depth]\label{thm:uniform}
There is an absolute constant $C$ such that, for all sufficiently large $n$, there is a connected $n$-vertex graph $G$ with
\[
    \psd(G)\ge n-C\sqrt n.
\]
Consequently,
\[
    \lim_{n\to\infty}\frac{\PS(n)}{n}=1.
\]
\end{theorem}

\par\noindent The proof is given in Appendix~\ref{app:uniform}.

The join padding creates automorphisms among the added universal vertices. The next construction removes that loss on an infinite sequence of orders while substantially improving the deficit.

\section{Many clocks and a polylogarithmic deficit}

Fix $r\ge4$ and distinct pairwise coprime integers
\[
    m_1,\ldots,m_r\ge6,
\]
and write
\[
    L=\prod_{j=1}^r m_j.
\]
Take disjoint clocks $Q_{m_1},\ldots,Q_{m_r}$. Add vertices
\[
    X=\{x_0,\ldots,x_{L-1}\}
\]
and, for each $k$, join $x_k$ to the $r$ vertices
\[
    a^{(m_j)}_{k\bmod m_j},\qquad j=1,\ldots,r.
\]
Let $H$ be the resulting graph and
\[
    G_i=H-\{x_0,\ldots,x_{i-1}\},
    \qquad 1\le i\le L-1.
\]

\begin{theorem}[Many-clock chain]\label{thm:many-clock}
With the construction above, every $G_i$ is connected and asymmetric. For $1\le i\le L-2$, the vertices $x_i$ and $x_{L-1}$ are pseudosimilar in $G_i$. Thus the connected asymmetric graph $G_1$ satisfies
\[
    |V(G_1)|=L+2\sum_{j=1}^r m_j-1
\]
and
\[
    \psd(G_1)\ge L-2
    =|V(G_1)|-2\sum_{j=1}^r m_j-1.
\]
Moreover, every graph encountered in this active chain is asymmetric.
\end{theorem}

\par\noindent The proof is given in Appendix~\ref{app:many-clock}.

Taking the moduli to be small distinct primes makes the product much larger than their sum. The standard prime estimates needed below follow, for example, from \citet{RosserSchoenfeld1962}.

\begin{corollary}[Polylogarithmic deficit]\label{cor:polylog}
There is an infinite sequence of connected asymmetric graphs $G$ such that
\[
    \psd(G)
    \ge |V(G)|-
    O\!\left(\frac{(\log |V(G)|)^2}{\log\log |V(G)|}\right).
\]
Every intermediate graph in the corresponding active chain is connected and asymmetric.
\end{corollary}

\par\noindent The proof is given in Appendix~\ref{app:polylog}.

\section{Relation to earlier pseudosimilarity constructions}

The construction above fits naturally into the group-theoretic picture of pseudosimilarity. Godsil and Kocay proved that finite pseudosimilar pairs can be understood through an automorphism of a larger graph whose orbit has been partially deleted \cite{GodsilKocay1982}. In our language, the ambient graph $H$ carries a long cyclic orbit $x_0,x_1,\ldots,x_{L-1}$. Passing to $G_i$ removes a consecutive block from that orbit. The intact cyclic motion no longer survives as an automorphism of $G_i$, but one step of that motion still gives an isomorphism between the two cards obtained by deleting the endpoints adjacent to the break. Thus the same hidden orbit generates pseudosimilarity at every level.

This differs from the classical problem of constructing a large mutually pseudosimilar set. In that problem one seeks many vertices $u_1,\ldots,u_k$ in a single graph with
\[
    G-u_1\cong\cdots\cong G-u_k
\]
while the $u_i$ lie in distinct automorphism orbits. Kimble, Schwenk, and Stockmeyer initiated systematic constructions of this type \cite{KimbleSchwenkStockmeyer1981}, Kocay related them to group actions on cosets \cite{Kocay1984}, and Lauri developed improved constructions and surveys of the extremal questions \cite{Lauri1996,Lauri1997,Lauri2003}. Our active chain may have only one distinguished pseudosimilar pair at a given level; what is large is the number of levels through which a pseudosimilar deletion remains available.

There is also a useful contrast with minimal pseudosimilarity. A graph containing a pseudosimilar pair can be minimal in the sense that no proper subgraph contains such a pair; this viewpoint already occurs in the structural theory of pseudosimilarity, particularly for trees \cite{HararyPalmer1966,KirkpatrickKlaweCorneil1983}. The recursive question studied here asks for the opposite behavior: graphs for which one can repeatedly move to a proper card while retaining an active deletion ambiguity. Theorem~\ref{thm:uniform} shows that this opposite behavior can persist for an asymptotically full number of deletions.

Finally, the parameter is distinct from deck overlap and reconstruction numbers. Large overlap asks whether two different graphs have many isomorphic cards, while $\psd(G)$ concerns nested cards of one graph. The constructions nevertheless share a common theme: a small structural core controls a much larger family of deletion equivalences. Recent common-card constructions of Ivanov exploit controlled symmetries of blown-up roles to create many cross-graph card isomorphisms \cite{Ivanov2026}; the present work is relevant to that line because our clock mechanism likewise concentrates many deletion isomorphisms into a small structural core, but does so along a nested sequence of cards within a single graph.

\section{Open problems}

The gap between the elementary upper bound \eqref{eq:trivial-upper} and our lower bounds is still substantial in additive terms. The most immediate question is whether the polylogarithmic remainder can be reduced to a constant.

\begin{problem}
Is there a constant $C$ and infinitely many graphs $G$ such that
\[
    \psd(G)\ge |V(G)|-C?
\]
Can this be achieved with $G$ connected and asymmetric, and with every intermediate graph asymmetric?
\end{problem}

The many-clock construction is naturally connected with the problem of realizing a cyclic motion of very large order on a small auxiliary set. For squarefree $L$, the clock sizes are the prime divisors of $L$, so the overhead reflects the familiar sum-versus-product economy of faithful permutation representations of cyclic groups. It would be interesting to know whether a different graph-theoretic realization of the hidden motion can beat this architecture asymptotically.

A second direction is to restrict the graph class.

\begin{problem}
Determine the maximum pseudosimilarity depth among trees, planar graphs, bounded-degree graphs, and regular graphs.
\end{problem}

Trees are particularly natural because the classical Harary--Palmer construction and the later characterization of Kirkpatrick, Klawe, and Corneil give strong structural information about pseudosimilar vertices in trees \cite{HararyPalmer1966,KirkpatrickKlaweCorneil1983,Lauri1997}. Our present construction uses cycles and vertices of growing degree, so it does not address this case.

A third question asks whether high depth forces a recognizable global structure. Godsil--Kocay theory shows that a single pseudosimilar pair is explained by a hidden automorphism in a suitable supergraph \cite{GodsilKocay1982}. In our examples one cyclic supergraph explains the entire chain simultaneously.

\begin{problem}
If $\psd(G)$ is linear in $|V(G)|$, must a positive fraction of an active chain be generated by a common ambient automorphism, or can long chains arise from unrelated local pseudosimilarities at successive levels?
\end{problem}

The latter possibility would represent a genuinely different mechanism from the clock construction.

\section{Conclusion}

Pseudosimilarity is usually viewed as a static ambiguity between two vertices of one graph. The active-depth viewpoint makes the ambiguity recursive. A broken cyclic orbit gives a particularly transparent mechanism: deleting one endpoint of the break moves the graph to the next level, where the new endpoint is again pseudosimilar to the fixed vertex on the other side. One clock gives a first linear construction, two clocks reduce the inactive part to square-root order, and many relatively prime clocks reduce it to polylogarithmic order along an infinite family. In particular, the maximum pseudosimilarity depth is asymptotic to the full order of the graph. This suggests that recursive deletion ambiguity is substantially more abundant than the static picture alone might indicate.

\clearpage
\appendix
\section{Deferred proofs}\label{app:proofs}

All proofs omitted from the main text are collected here in the order in which their statements appear.

\subsection{Proof of the elementary upper bound}\label{app:upper-bound}

\begin{proof}[Equation~\eqref{eq:trivial-upper}]
Immediately before the last active deletion, the current graph must contain a pseudosimilar pair.  The smallest graph containing such a pair has eight vertices.  Hence at most $n-7$ active deletions can occur in an $n$-vertex graph.
\end{proof}

\subsection{Proof of the sliding-gap lemma}\label{app:sliding-gap}

\begin{proof}[Lemma~\ref{lem:sliding-gap}]
The graph $G_i-x_i$ is obtained from $H$ by deleting
\[
    \{x_0,x_1,\ldots,x_i\},
\]
whereas $G_i-x_{L-1}$ is obtained by deleting
\[
    \{x_{L-1},x_0,x_1,\ldots,x_{i-1}\}.
\]
Applying $\rho$ to the second deleted set gives the first. Hence
\[
    G_i-x_{L-1}\cong G_i-x_i.
\]
\end{proof}

\subsection{Verification of the 17-vertex example}\label{app:example-seventeen}

\begin{proof}[Example~\ref{ex:seventeen}]
Connectivity is immediate.  Before $x_0$ is removed, simultaneous rotation of all subscripts is an automorphism.  For $1\le i\le4$, the rotation by one maps the deleted set
\[
    \{x_5,x_0,\ldots,x_{i-1}\}
\]
to
\[
    \{x_0,x_1,\ldots,x_i\}.
\]
Consequently $G_i-x_5\cong G_i-x_i$, so $x_i$ and $x_5$ are removal-similar.

It remains to rule out an automorphism exchanging them.  The surviving $x$-vertices have degree $1$, the $b$-vertices have degree $3$, and the $a$-vertices have degree at least $5$, so every automorphism preserves these three sets.  Its restriction to the graph on the $a$- and $b$-vertices is a rotation; this is the case $m=6$ of Lemma~\ref{lem:clock-aut}.  Such a rotation must preserve the surviving index interval $\{i,i+1,\ldots,5\}$, whose translation stabilizer in $\mathbb Z/6\mathbb Z$ is trivial.  Hence every $G_i$ is asymmetric.  The removal-similar vertices $x_i,x_5$ are therefore pseudosimilar for $i=1,2,3,4$, proving the claimed active chain.
\end{proof}

\subsection{Proof of the clock automorphism lemma}\label{app:clock-aut}

\begin{proof}[Lemma~\ref{lem:clock-aut}]
Every $a_i$ has degree $5$ and every $b_i$ has degree $3$, so every automorphism preserves $A$ and $B$. Its restriction to the cycle on $A$ is therefore a rotation or a reflection.

Every rotation extends to $Q_m$. It remains to exclude reflections. Write
\[
    S=\{0,1,3\}\subseteq \Z/m\Z.
\]
The $B$-neighborhoods are precisely the translates $i+S$. Under a reflection they would become translates of $-S=\{0,-1,-3\}$. Thus a reflection could extend only if $-S$ were a translate of $S$. Since $0\in -S$, the translating element would have to be one of $0,-1,-3$. A direct check shows that none of these gives equality when $m\ge6$. Hence no reflection extends, and the rotations are all the automorphisms.
\end{proof}

\subsection{Proof of the one-clock theorem}\label{app:one-clock}

\begin{proof}[Theorem~\ref{thm:one-clock}]
Removal-similarity follows from Lemma~\ref{lem:sliding-gap}.

It remains to prove asymmetry. The surviving vertices of $X$ have degree $1$, while the vertices of $B$ have degree $3$ and the vertices of $A$ have degree at least $5$. Hence every automorphism of $G_i$ preserves the three sets $A$, $B$, and $X\cap V(G_i)$. By Lemma~\ref{lem:clock-aut}, its restriction to $Q_m$ is a rotation by some $d\in\Z/m\Z$. Therefore it sends each surviving $x_j$ to $x_{j+d}$.

The surviving indices form the cyclic interval
\[
    I_i=\{i,i+1,\ldots,m-1\}.
\]
A proper nonempty cyclic interval has trivial stabilizer under translations of $\Z/m\Z$, so $I_i+d=I_i$ implies $d=0$. The automorphism is therefore the identity. Thus $G_i$ is asymmetric, and the removal-similar vertices $x_i,x_{m-1}$ are pseudosimilar.

Starting from $G_1$, successively delete
\[
    x_1,x_2,\ldots,x_{m-2}.
\]
At the moment $x_i$ is deleted, its pseudosimilar mate is $x_{m-1}$. This gives an active chain of length $m-2$.
\end{proof}

\subsection{Proof of the two-clock theorem}\label{app:two-clock}

\begin{proof}[Theorem~\ref{thm:two-clock}]
The graph is connected because every surviving $x_k$ joins the two connected clock components. Removal-similarity of $x_i$ and $x_{L-1}$ again follows from Lemma~\ref{lem:sliding-gap}.

For asymmetry, note that the surviving $x$-vertices have degree $2$, the $B$-vertices of the clocks have degree $3$, and every $A$-vertex has degree at least $5$. Hence an automorphism preserves $X$ and the two clock subgraphs. Since the clock orders $2p$ and $2q$ are different, the two clocks cannot be interchanged. By Lemma~\ref{lem:clock-aut}, the restrictions are rotations, say by $s$ modulo $p$ and by $t$ modulo $q$.

By the Chinese remainder theorem there is a unique $d\pmod{L}$ satisfying
\[
    d\equiv s\pmod p,
    \qquad
    d\equiv t\pmod q.
\]
The neighborhood code of $x_k$ therefore forces any surviving $x_k$ to map to $x_{k+d}$. Hence the surviving index set
\[
    I_i=\{i,i+1,\ldots,L-1\}
\]
must satisfy $I_i+d=I_i$. Its translation stabilizer is trivial, so $d=0$, whence $s=t=0$. Thus every automorphism is the identity.

The active chain is
\[
    x_1,x_2,\ldots,x_{L-2},
\]
with $x_{L-1}$ serving as a pseudosimilar mate throughout.
\end{proof}

\subsection{Proof of the square-root corollary}\label{app:sqrt-sequence}

\begin{proof}[Corollary~\ref{cor:sqrt-sequence}]
Apply Theorem~\ref{thm:two-clock} with $p=m$ and $q=m+1$.
\end{proof}

\subsection{Proof of the join-padding lemma}\label{app:join-padding}

\begin{proof}[Lemma~\ref{lem:join-padding}]
An isomorphism $F-u\cong F-v$ extends by the identity on $K_t$. The newly added vertices are precisely the universal vertices of the join, so every automorphism preserves them setwise and restricts to an automorphism of $F$. Thus an automorphism of the join mapping $u$ to $v$ would induce one in $F$, contrary to pseudosimilarity.
\end{proof}

\subsection{Proof of the uniform asymptotic theorem}\label{app:uniform}

\begin{proof}[Theorem~\ref{thm:uniform}]
Choose the largest $m$ for which $n_m=m^2+5m+1\le n$, and let $t=n-n_m$. Apply Lemma~\ref{lem:join-padding} to the graph in Corollary~\ref{cor:sqrt-sequence}. Since
\[
    n_{m+1}-n_m=2m+6,
\]
we have $t<2m+6$. Hence
\[
\begin{aligned}
    n-\psd(G)
    &\le t+4m+3\\
    &<6m+9
     =O(\sqrt n).
\end{aligned}
\]
The upper bound $\PS(n)\le n-7$ from \eqref{eq:trivial-upper} then gives the limit.
\end{proof}

\subsection{Proof of the many-clock theorem}\label{app:many-clock}

\begin{proof}[Theorem~\ref{thm:many-clock}]
Connectivity and removal-similarity are as before. We prove asymmetry.

Let $B$ be the union of the $B$-layers of all clocks. These are exactly the degree-$3$ vertices: each $x_k$ has degree $r\ge4$, while every $A$-vertex has degree at least $5$. Hence $B$ is invariant under automorphisms. The union $A$ of the $A$-layers is then recognizable as the set of vertices outside $B$ that have a neighbor in $B$; the $x$-vertices have no neighbors in $B$. Thus $X$ is invariant as well.

The induced graph on $A\cup B$ is the disjoint union of the clocks $Q_{m_j}$. Their orders are distinct, so each clock is preserved individually. Lemma~\ref{lem:clock-aut} implies that the restriction to $Q_{m_j}$ is a rotation by some $s_j\pmod{m_j}$. By the Chinese remainder theorem there is a unique $d\pmod L$ satisfying
\[
    d\equiv s_j\pmod{m_j}
    \qquad (j=1,\ldots,r).
\]
The neighborhood code of $x_k$ then forces $x_k\mapsto x_{k+d}$ whenever both vertices survive. Therefore the surviving interval
\[
    I_i=\{i,i+1,\ldots,L-1\}
\]
is invariant under translation by $d$. Again its stabilizer is trivial, so $d=0$ and all $s_j=0$. The automorphism is the identity.

The rest follows by deleting $x_1,x_2,\ldots,x_{L-2}$ in order and using $x_{L-1}$ as the mate at every stage.
\end{proof}

\subsection{Proof of the polylogarithmic-deficit corollary}\label{app:polylog}

\begin{proof}[Corollary~\ref{cor:polylog}]
Take $m_1,\ldots,m_r$ to be the first $r$ primes at least $7$. Put
\[
    L_r=\prod_{j=1}^r m_j,
    \qquad
    S_r=\sum_{j=1}^r m_j.
\]
The prime number theorem gives
\[
    \log L_r=\Theta(r\log r),
    \qquad
    S_r=\Theta(r^2\log r).
\]
Hence
\[
    S_r=O\!\left(\frac{(\log L_r)^2}{\log\log L_r}\right).
\]
Theorem~\ref{thm:many-clock} gives graphs of order $L_r+2S_r-1$ and depth at least $L_r-2$. Since $S_r=o(L_r)$, replacing $L_r$ by the graph order inside the logarithms changes only the implicit constant.
\end{proof}

\end{document}